\documentclass[12pt]{amsart}
\usepackage[colorlinks=true, pdfstartview=FitH, linkcolor=blue, citecolor=blue, urlcolor=blue]{hyperref}
\usepackage{fullpage}
\usepackage{lineno}
\usepackage{enumerate}
\usepackage{bbm}
\usepackage[english]{babel}
\usepackage[utf8]{inputenc}
\usepackage{amsmath}
\usepackage{graphicx}
\usepackage[colorinlistoftodos]{todonotes}
\usepackage{amsmath, amsthm, amssymb}
\usepackage[bottom]{footmisc}
\usepackage{cite}
\usepackage{graphicx}
\usepackage{yfonts}
\usepackage{longtable}
\graphicspath{{figures/}} 
\usepackage{mathtools}
\usepackage{enumitem}
\usepackage[T1]{fontenc}
\usepackage{multirow}

\newtheorem{theorem}{Theorem}
\newtheorem{corollary}[theorem]{Corollary}
\newtheorem{lemma}[theorem]{Lemma}
\newtheorem{proposition}[theorem]{Proposition}

\theoremstyle{definition}
\newtheorem{definition}[theorem]{Definition}
\newtheorem{remark}[theorem]{Remark}

\newcommand{\bC}{\mathbb{C}}
\newcommand{\bE}{\mathbb{E}}

\newcommand{\bR}{\mathbb{R}}
\newcommand{\cF}{\mathcal{F}}
\newcommand{\cN}{\mathcal{N}}

\makeatletter
\@namedef{subjclassname@2020}{\textup{2020} Mathematics Subject Classification}
\makeatother

\begin{document}


\title[Distribution of prime factors with a given multiplicity]{Distribution of the number of prime factors with a given multiplicity}
\author{Ertan Elma}
\address{Department of Mathematics \& Computer Science, University of Lethbridge, Faculty of Arts and Science, 4401 University Drive, Lethbridge, Alberta T1K 3M4, Canada}
\email{ertan.elma@uleth.ca}
\author{Greg Martin}
\address{Department of Mathematics, University of British Columbia, Vancouver, British Columbia V6T 1Z2, Canada}
\email{gerg@math.ubc.ca}
\maketitle

\begin{abstract}
Given an integer $k\ge2$, let $\omega_k(n)$ denote the number of primes that divide~$n$ with multiplicity exactly~$k$. We compute the density~$e_{k,m}$ of those integers~$n$ for which $\omega_k(n)=m$ for every integer $m\ge0$. We also show that the generating function $\sum_{m=0}^\infty e_{k,m}z^m$ is an entire function that can be written in the form $\prod_{p} \bigl(1+{(p-1)(z-1)}/{p^{k+1}} \bigr)$; from this representation we show how to both numerically calculate the~$e_{k,m}$ to high precision and provide an asymptotic upper bound for the~$e_{k,m}$. We further show how to generalize these results to all additive functions of the form $\sum_{j=2}^\infty a_j \omega_j(n)$; when $a_j=j-1$ this recovers a classical result of R\'enyi on the distribution of $\Omega(n)-\omega(n)$.
\end{abstract}

\section{Introduction}

Let $\omega(n)$ be the number of distinct prime factors of a positive integer~$n$, and let $\Omega(n)$ be the number of prime factors of~$n$ counted with multiplicity. Average behaviours of such arithmetic functions are understood via their summatory functions. It is known~\cite{Hardy_Ramanujan} (see also~\cite[Theorems 427--430]{Hardy_Wright}) that
\begin{align}
\begin{split} \label{summatory_of_omegas}
\sum_{n\leqslant x} \omega(n)&=x\log\log x+bx+O\biggl(\frac{x}{\log x} \biggr) \\
\sum_{n\leqslant x} \Omega(n)&=x\log\log x+\biggl(b+\sum_{p } \frac{1}{p(p-1)} \biggr)x+O\biggl(\frac{x}{\log x} \biggr);
\end{split}
\end{align}
here the constant~$b$ is defined by
\begin{align} \label{b}
b&=\gamma_0+\sum_{p} \sum_{j=2}^{\infty} \frac{1}{jp^j}
\end{align}
where $\gamma_0$ denotes the Euler--Mascheroni constant. (In this paper, $\sum_p$ and $\prod_p$ always denote sums and products running over all prime numbers.) The celebrated Erd\H os--Kac theorem tells us that both $\omega(n)$ and $\Omega(n)$ can be normalized to have Gaussian limiting distribution functions.

By the asymptotic formulas~\eqref{summatory_of_omegas}, the difference $\Omega(n)-\omega(n)$ has an average value, namely the constant
\begin{align*}
\lim_{x\to\infty} \frac1x \sum_{n\leqslant x} \bigl( \Omega(n)-\omega(n) \bigr) = \sum_{p} \frac{1}{p(p-1)},
\end{align*}
which provides motivation to study the frequency of each possible value of $\Omega(n)-\omega(n)$. For any integer $m\geqslant 0$, define
\begin{align*}
\cN_m(x)=\lbrace n\leqslant x\colon \Omega(n)-\omega(n)=m\rbrace.
\end{align*}
R\'{e}nyi~\cite{Renyi} (see also~\cite[Section 2.4]{MV}) proved that the (natural) densities
\begin{align} \label{renyi}
d_m=\lim_{x\to \infty} \frac{\#\cN_m(x)}{x} = \frac{6}{\pi^2} \sum_{\substack{f\in\cF \\ \Omega(f)-\omega(f)=m}} \frac{1}{f} \prod_{p\mid f} \biggl(1+\frac{1}{p} \biggr)^{-1}
\end{align}
exist for every $m\geqslant0$, where $\cF$ is the set of powerful numbers (the set of positive integers all of whose prime factors have multiplicity $\geqslant 2$).
Furthermore, he showed that these densities have the generating function
\begin{align} \label{generating_series_Renyi}
\sum_{m=0}^{\infty}d_mz^m=\prod_{p} \biggl(1-\frac{1}{p} \biggr)\biggl(1+ \frac{1}{p-z} \biggr) \qquad (|z|<2).
\end{align}
(Note the special case $d_0 = \prod_p (1-\frac1p)(1+\frac1p) = \frac1{\zeta(2)} = \frac6{\pi^2}$ for the density of squarefree numbers, which can also be confirmed by realizing that the sum in equation~\eqref{renyi} contains only the single term $f=1$ when $m=0$.)
In particular, the smaller function $\Omega(n)-\omega(n)$ already has a (discrete) limiting distribution function, without needing normalization in the way that the larger functions $\omega(n)$ and $\Omega(n)$ individually do.

As a refinement of the function $\omega(n)$, Liu and the first author introduced the functions
\begin{align*}
\omega_k(n)=\sum_{p^k \| n}1
\end{align*}
for each integer $k\geqslant1$, so that $\omega_k(n)$ counts the number of prime factors of~$n$ with multiplicity~$k$ and thus $\omega(n) = \sum_{k=1}^\infty \omega_k(n)$. They showed~\cite{Elma_Liu} that
\begin{align*}
\sum_{n\leqslant x} \omega_1(n)=x\log\log x+\biggl(b-\sum_{p} \frac{1}{p^2} \biggr)x+O\biggl(\frac{x}{\log x} \biggr)
\end{align*}
where~$b$ is the constant from equation~\eqref{b}, while
\begin{align} \label{first_moment_omega_k2}
\sum_{n\leqslant x} \omega_k(n)=x\sum_{p} \frac{p-1}{p^{k+1}}+O\bigl(x^{(k+1)/(3k-1)} \log^2x\bigr) \qquad (k\geqslant 2).
\end{align}
They also showed that the larger function $\omega_1(n)$ has a Gaussian limiting distribution function after being normalized in the same way as the classical $\omega(n)$ and $\Omega(n)$. However, since equation~\eqref{first_moment_omega_k2} shows that $\omega_k(n)$ has an average value for each $k\geqslant2$, we might expect these smaller functions to have limiting distributions without needing to be normalized.

In this paper, we obtain the limiting distribution for the functions $\omega_k(n)$ for $k\geqslant 2$, analogous to the results of R\'{e}nyi described above. For integers $m\geqslant 0$, define
\begin{align} \label{the_set_N}
\cN_{k,m}(x)=\lbrace n\leqslant x\colon \omega_k(n)=m\rbrace
\end{align}
to be the set of positive integers $n\leqslant x$ with exactly~$m$ prime factors of multiplicity~$k$. Our main result establishes the existence of the densities
\begin{align*}
e_{k,m} = \lim_{x\to \infty} \frac{\#\cN_{k,m}(x)}{x}
\end{align*}
and provides a closed-form expression for them.

\begin{theorem} \label{main_result_1}
Uniformly for all integers $k\geqslant 2$ and $m\geqslant 0$,
\begin{align*}
\#\cN_{k,m}(x) = e_{k,m}x+O(x^{1/2} \log x)
\end{align*}
with
\begin{align*}
e_{k,m}= \frac{6}{\pi^2} \sum_{\substack{f\in\cF \\\omega_k(f)=m}} \frac{1}{f} \prod_{p\mid f} \biggl(1+\frac{1}{p} \biggr)^{-1}.
\end{align*}
\end{theorem}

\begin{remark} \label{abs conv remark}
Note that the $e_{k,m}$ are all nonnegative, and we can check that they do sum to~$1$:
\begin{align*}
\sum_{m=0}^\infty e_{k,m} &= \sum_{m=0}^\infty \frac{6}{\pi^2} \sum_{\substack{f\in\cF \\\omega_k(f)=m}} \frac{1}{f} \prod_{p\mid f} \biggl(1+\frac{1}{p} \biggr)^{-1} = \frac{6}{\pi^2} \sum_{f\in\cF} \frac{1}{f} \prod_{p\mid f} \biggl(1+\frac{1}{p} \biggr)^{-1}.
\end{align*}
Since the summand is a multiplicative function of~$f$, as is the indicator function of~$\cF$, the right-hand side equals its Euler product
\[
\frac{6}{\pi^2} \prod_p \biggl( 1 + 0 + \biggl(1+\frac{1}{p} \biggr)^{-1} \biggl( \frac1{p^2} + \frac1{p^3} + \cdots \biggr) \biggr) = \frac{6}{\pi^2} \prod_p \biggl( 1-\frac1{p^2} \biggr)^{-1} = 1.
\]
The same remark applies to the densities in equation~\eqref{formula_general_density} below.
\end{remark}

Moreover, we obtain an identity analogous to equation~\eqref{generating_series_Renyi} for the generating function of the densities $e_{k,m}$ for fixed $k\geqslant 2$, from which we can derive an upper bound for the densities $e_{k,m}$ when $k\geqslant 2$ is fixed and $m\to \infty$.

\begin{theorem} \label{main_result_2}
Let $k\geqslant 2$ be an integer. For all $z\in\bC$ with $|z|\leqslant1$,
\begin{equation} \label{main_eq_2}
\sum_{m=0}^{\infty}e_{k,m}z^{m}=\prod_{p} \biggl(1+\frac{(p-1)(z-1)}{p^{k+1}} \biggr).
\end{equation}
\end{theorem}

\begin{corollary} \label{main_cor_2}
For each fixed $k\geqslant2$, we have $e_{k,m} \leqslant m^{-(k-o(1))m}$ as $m\to\infty$.
\end{corollary}

\begin{remark}
The proof of the upper bound in Corollary~\ref{main_cor_2} (see Section~\ref{section_for_main_results_2_and_4}) shows that for each $k\geqslant2$, the bound is attained for infinitely many~$m$; it would be interesting to try to show that $e_{k,m} = m^{-(k-o(1))m}$ for all~$k$ and~$m$. Moreover, the corollary and its proof show that both sides of equation~\eqref{main_eq_2} converge to entire functions, and thus Theorem~\ref{main_result_2} actually holds for all $z\in\bC$ by uniqueness of analytic continuation. The same remarks apply to the generating functions in Corollary~\ref{main_result_4seo} and the upper bounds in Corollary~\ref{main_cor_4} below.
\end{remark}

Some numerical values of $e_{k,m}$ are given in Table~\ref{ekm table}. The numbers in the first column corresponding to $m=0$ are increasing as $k$ increases, whereas the numbers in other columns are decreasing. This behaviour stems from the fact that the case $m=0$ indicates the nonexistence of prime factors with multiplicity~$k$, which becomes more probable as~$k$ increases. (Note also that each number in the first column exceeds $\frac6{\pi^2} \approx 0.608$, since every squarefree number~$n$ certainly has $\omega_k(n)=0$ for all $k\geqslant2$.) On the other hand, for $m\geqslant 1$, the criterion $\omega_{k}(n)=m$ indicates the existence of prime factors with multiplicity~$k$, which becomes less probable as~$k$ increases. Details of the calculations of these values are given in Section~\ref{section_calculations}, although we do note here that the calculations use the generating function in Theorem~\ref{main_result_2} rather than the formula for $e_{k,m}$ in Theorem~\ref{main_result_1}.

\begin{table}[hbt]
\caption{Some values of $e_{k,m}$}
\label{ekm table}
\begin{tabular}{|c|cccc|}
\hline
& $m=0$ & $m=1$ & $m=2$ & $m=3$ \\
\hline
$k=2$ & 0.748535831 & 0.226618489 & 0.023701061 & 0.001117529 \\
\hline
$k=3$ & 0.904708927 & 0.092831692 & 0.002440388 & 0.000018941 \\
\hline
$k=4$ & 0.959088654 & 0.040585047 & 0.000325821 & 0.000000477 \\
\hline
$k=5$ & 0.981363751 & 0.018587581 & 0.000048654 & 0.000000014 \\
\hline
\end{tabular}
\end{table}

A consequence of Theorem~\ref{main_result_1} and Remark~\ref{abs conv remark} is that $\omega_k(n)$ has a limiting distribution, which is the same as the distribution of the nonnegative integer-valued random variable $X_k$ that takes the value~$m$ with probability~$e_{k,m}$. While it is straightforward to calculate the expectation and variance of this limiting distribution via the expressions
\[
\lim_{x\to\infty} \frac1x \sum_{n\leqslant x} \omega_k(n) \qquad\text{and}\qquad \lim_{x\to\infty} \frac1x \sum_{n\leqslant x} \omega_k(n)^2 - \biggl( \lim_{x\to\infty} \frac1x \sum_{n\leqslant x} \omega_k(n) \biggr)^2,
\]
we can observe that the generating function from Theorem~\ref{main_result_2} provides a quick way to obtain the answers with no further input from number theory.

\begin{corollary} \label{corollary_X}
The limiting distribution of $\omega_k(n)$ has expectation $\displaystyle\sum_{p} \frac{p-1}{p^{k+1}}$ and variance $\displaystyle\sum_p \frac{p-1}{p^{k+1}} \biggl( 1 - \frac{p-1}{p^{k+1}} \biggr)$.
\end{corollary}

\begin{remark}
Not surprisingly, these quantities are the expectation and variance of the sum of infinitely many Bernoulli random variables $B_p$, indexed by primes~$p$, where $B_p$ takes the value~$1$ with probability $(p-1)/p^{k+1}$ (the density of those integers exactly divisible by~$p^k$).
\end{remark}

These quantities are easy to calculate to reasonably high precision (see Section~\ref{section_calculations} for details); we record some numerical values in Table~\ref{E and var table}. The reader can confirm that the listed expectations are in good agreement with the quantities $0e_{k,0}+1e_{k,1}+2e_{k,2}+3e_{k,3}$ as calculated from Table~\ref{ekm table}.

\begin{table}[hbt]
\caption{Statistics of the limiting distribution of $\omega_k(n)$}
\label{E and var table}
\begin{tabular}{|c|c|c|}
\hline
& expectation of $\omega_k(n)$ & variance of $\omega_k(n)$ \\
\hline
$k=2$ & 0.277484775 & 0.254931583 \\
\hline
$k=3$ & 0.097769500 & 0.093205673 \\
\hline
$k=4$ & 0.041238122 & 0.040192048 \\
\hline
$k=5$ & 0.018684931 & 0.018433195 \\
\hline
\end{tabular}
\end{table}

\subsection{Generalizations}

It turns out that our proof of Theorem~\ref{main_result_1} goes through for a far larger class of additive functions than just the $\omega_k(n)$. Given any sequence $A=(a_1,a_2, a_3 ,\dots)$ of complex numbers, define the additive function
\begin{equation} \label{omega A def}
\omega_{A}(n)=\sum_{j=1}^{\infty}a_j\omega_{j}(n),
\end{equation}
which is of course a finite sum for each integer~$n$.

\begin{remark} \label{special cases}
This definition generalizes all the examples we have seen so far: 
\begin{itemize}
\item if $a_j=1$ always then $\omega_{A}(n) = \omega(n)$;
\item if $a_j=j$ always then $\omega_{A}(n) = \Omega(n)$;
\item if $a_j=j-1$ always then $\omega_{A}(n)=\Omega(n)-\omega(n)$;
\item for a fixed positive integer~$k$, if $a_k=1$ while $a_j=0$ for $j\neq k$, then $\omega_{A}(n)=\omega_k(n)$.
\end{itemize}
\end{remark}

When $a_1\ne 0$, classical techniques show that the large function $\frac1{a_1} \omega_A(n)$ has the same Gaussian limiting distribution as $\omega(n)$ and $\Omega(n)$ when properly normalized (at least if the~$a_j$ do not grow too quickly). Therefore we restrict our attention to the smaller functions $\omega_A(n)$ where $a_1=0$, which we expect to have limiting distributions without needing normalization.

For $m\in\bC$, define
\begin{align*}
\cN_{A,m}(x)=\lbrace n\leqslant x \colon \omega_{A}(n)=m\rbrace .
\end{align*}
Our next result, which generalizes both equation~\eqref{renyi} and Theorem~\ref{main_result_1}, establishes the existence of the densities
\begin{align} \label{defn_general_density}
e_{A,m} = \lim_{x\to \infty} \frac{\#\cN_{A,m}(x)}{x}
\end{align}
and provides a closed-form expression for them.

\begin{theorem} \label{main_result_3}
Uniformly for all sequences $A=(0,a_2,a_3,\dots)$ of complex numbers with $a_1=0$ and for all $m\in\bC$,
\begin{align*}
\#\cN_{A,m}(x) = e_{A,m}x+O(x^{1/2} \log x)
\end{align*}
with
\begin{align} \label{formula_general_density}
e_{A,m}= \frac{6}{\pi^2} \sum_{\substack{f\in\cF \\ \omega_A(f) = m}} \frac{1}{f} \prod_{p\mid f} \biggl(1+\frac{1}{p} \biggr)^{-1}.
\end{align}
\end{theorem}

If we now restrict to the case where the~$a_j$ (and thus all values of $\omega_A(n)$) are nonnegative integers, it once again makes sense to consider generating functions.
Our next result generalizes both equation~\eqref{generating_series_Renyi} and Theorem~\ref{main_result_2} in light of Remark~\ref{special cases}.

\begin{theorem} \label{main_result_4}
Let $A=(0,a_2,a_3,\dots)$ be a sequence of nonnegative integers. For all $z\in\bC$ with $|z|\leqslant1$,
\begin{align} \label{generating_series_general}
\sum_{m=0}^{\infty}e_{A,m}z^{m}=\prod_{p} \biggl(1-\frac{1}{p^2}+\sum_{j=2}^{\infty}z^{a_j} \biggl(\frac{1}{p^{j}}-\frac{1}{p^{j+1}} \biggr)\biggr).
\end{align}
\end{theorem}

Again this theorem shows that $\omega_A(n)$ has a limiting distribution when the $a_j$ are nonnegative integers, and we can therefore generalize Corollary~\ref{corollary_X}; we record only the expectation for simplicity.

\begin{corollary} \label{corollary_Y}
Let $A=(0,a_2,a_3,\dots)$ be a sequence of nonnegative integers. The limiting distribution of $\omega_A(n)$ has expectation $\displaystyle\sum_{p} \sum_{j=2}^{\infty}a_j\biggl(\frac{1}{p^{j}}-\frac{1}{p^{j+1}} \biggr)$.
\end{corollary}

\begin{remark}
It is certainly possible for this expectation to be infinite, as the example $A=(0,2,4,8,16,\dots)$ shows. In such cases $\frac1x
\sum_{n\leqslant x} \omega_A(n)$ grows too quickly for the mean value of $\omega_A(n)$ to exist. Note, however, that Theorems~\ref{main_result_3} and~\ref{main_result_4} hold no matter how quickly the sequence~$A$ might grow.
\end{remark}

We examine three specific examples of such sequences for the purposes of illustration: set $S=(0,1,1,\dots)$ and $E = (0,1,0,1,\dots)$ and $O = (0,0,1,0,1,0,1,\dots)$. Then the corresponding omega functions are
\begin{equation*}
\omega_{{S}}(n) =\sum_{j\geqslant 2} \omega_j(n) \quad\text{and} \quad
\omega_{{E}}(n) =\sum_{\substack{j\geqslant 2\\ j \text{ even}}} \omega_{j}(n) \quad\text{and} \quad
\omega_{{O}}(n) =\sum_{\substack{j\geqslant 3\\ j \text{ odd}}} \omega_{j}(n)
\end{equation*}
which count, respectively, the number of primes dividing the powerful part of~$n$ (that is, the number of primes dividing~$n$ at least twice), the number of primes dividing~$n$ with even multiplicity, and the number of primes dividing~$n$ with odd multiplicity exceeding~$1$. For integers $m\geqslant 0$, let $e_{{S},m}$ and $e_{{E},m}$ and $e_{{O},m}$ be the corresponding densities defined in equation~\eqref{defn_general_density}. An easy calculation of the right-hand side of equation~\eqref{generating_series_general} in these cases (for which each factor becomes a geometric series) yields the following generating functions:

\begin{corollary} \label{main_result_4seo}
For all $z\in\bC$ with $|z|\leqslant1$,
\begin{align*}
\sum_{m=0}^{\infty}e_{{S},m}z^{m}&=\prod_{p} \biggl(1+\frac{z-1}{p^{2}} \biggr)
\\\sum_{m=0}^{\infty}e_{{E},m}z^{m}&=\prod_{p} \biggl(1+\frac{z-1}{p(p+1)} \biggr)
\\\sum_{m=0}^{\infty}e_{{O},m}z^{m}&=\prod_{p} \biggl(1+\frac{z-1}{p^{2}(p+1)} \biggr).
\end{align*}
\end{corollary}

\begin{corollary} \label{main_cor_4}
For each fixed $k\geqslant2$, we have $e_{{S},m} \leqslant m^{-(2-o(1))m}$ and $e_{{E},m} \leqslant m^{-(2-o(1))m}$ and $e_{{O},m} \leqslant m^{-(3-o(1))m}$ as $m\to\infty$.
\end{corollary}

\begin{remark}
One interesting class of functions for which our methods accomplish less than desired are functions of the form $\omega_A(n)$ where~$A$ contains integers but not necessarily only nonnegative integers. For example, if $A = (0,1,-1,0,0,\dots)$ then $\omega_A(n) = \omega_2(n)-\omega_3(n)$, while if $A=(0,1,-1,1,-1,\dots)$ then $\omega_A(n) = \omega_E(n) - \omega_O(n)$. The target $m=0$ is natural to investigate, as $\omega_A(n) = 0$ in these two examples translates into $\omega_2(n) = \omega_3(n)$ and $\omega_E(n) = \omega_O(n)$, respectively. While Theorem~\ref{main_result_3} gives a formula for the density of those integers~$n$ satisfying each of these equalities, our numerical techniques in Section~\ref{section_calculations} (which ultimately rely on being able to find the values of the derivatives of the appropriate generating function at $z=0$) are not able to approach the question of good numerical approximations to these densities.
\end{remark}

In Section~\ref{section_for_main_result_3} we establish Theorems~\ref{main_result_3} and~\ref{main_result_4}, the formula and generating function for $e_{A,m}$, from which Theorems~\ref{main_result_1} and~\ref{main_result_2} follow as special cases.
In Section~\ref{section_for_main_results_2_and_4} we deduce Corollaries~\ref{main_cor_2} and~\ref{main_cor_4} (the decay rates of $e_{k,m}$ and certain variants) from Theorem~\ref{main_result_2} and Corollary~\ref{main_result_4seo}.
Finally, in Section~\ref{section_calculations} we describe the computations leading to the numerical values in Tables~\ref{ekm table} and~\ref{E and var table}, as well as establishing Corollaries~\ref{corollary_X} and~\ref{corollary_Y} concerning the expectation and variance of the additive functions under examination.

\section{Closed form and generating function for the densities} \label{section_for_main_result_3}

We first prove Theorem~\ref{main_result_3}, which will also establish the special case that is Theorem~\ref{main_result_1}, by following the exposition of R\'{e}nyi's result~\eqref{renyi} in~\cite[Section 2.4]{MV}.
Recall the notation of equation~\eqref{omega A def}, and recall that~$\cF$ denotes the set of powerful numbers.

\begin{lemma}  \label{f<x error}
Uniformly for all sequences $A=(a_1,a_2,\dots)$ of complex numbers and all $m\in\bC$,
\[
\sum_{\substack{f\leqslant x\\f\in\cF \\\omega_A(f)=m}} \frac{1}{f^{1/2}} \prod_{p\mid f} (1-p^{-1/2} )^{-1} \ll \log x.
\]
\end{lemma}

\begin{proof}
By dropping the condition $\omega_{A}(f)=m$ and noting that $f\leqslant x$ implies that all prime factors of~$f$ are at most~$x$, we have by positivity
\begin{align*}
\sum_{\substack{f\leqslant x \\ f\in\cF \\\omega_A(f)=m}} \frac{1}{f^{1/2}} \prod_{p\mid f} (1-p^{-1/2} )^{-1}&\leqslant \sum_{\substack{f\leqslant x \\ f\in\cF}} \frac{1}{f^{1/2}} \prod_{p\mid f} (1-p^{-1/2} )^{-1}
 \leqslant \sum_{\substack{f\in\cF \\ p\mid f \implies p\leqslant x}} \frac{1}{f^{1/2}} \prod_{p\mid f} (1-p^{-1/2} )^{-1}.
\end{align*}
The right-hand side has an Euler product whose factors involve geometric series with common ratio~$p^{-1/2}$:
\begin{align*}
\sum_{\substack{f\in\cF \\ p\mid f \implies p\leqslant x}} \frac{1}{f^{1/2}} \prod_{p\mid f} (1-p^{-1/2} )^{-1} &= \prod_{p\leqslant x} \biggl(1 + \frac{(1-p^{-1/2} )^{-1}}{(p^2)^{1/2}} + \frac{(1-p^{-1/2} )^{-1}}{(p^3)^{1/2}} + \cdots \biggr)
\\&= \prod_{p\leqslant x} \biggl( 1 + \frac1{(p^{1/2}-1)^2} \biggr)
\\&= \prod_{p\leqslant x} \biggl( 1 - \frac1p \biggr)^{-1} \prod_{p\leqslant x} \biggl( 1 + \frac2{p(p^{1/2}-1)} \biggr);
\end{align*}
this establishes the lemma, since the first product is asymptotic to a multiple of $\log x$ as shown by Mertens, while the second is a convergent product of the form $\prod_p (1+O(p^{-3/2}))$. 
\end{proof}

\begin{lemma}  \label{f>x error}
Uniformly for all sequences $A=(a_1,a_2,\dots)$ of complex numbers and all $m\in\bC$,
\[
\sum_{\substack{f>x \\ f\in\cF \\\omega_A(f)=m}} \frac{1}{f} \prod_{p\mid f} \biggl(1+\frac{1}{p} \biggr)^{-1} \ll x^{-1/2}.
\]
\end{lemma}

\begin{proof}
Golomb~\cite{Golomb} proved that the number of powerful numbers up to~$y$ is asymptotic to a constant times~$y^{1/2}$. Thus for each integer $r\geqslant 0$,
\[
\sum_{\substack{ 2^rx < f \leqslant 2^{r+1}x \\f\in\cF \\\omega_A(f)=m}} \frac{1}{f} \prod_{p\mid f} \biggl(1+\frac{1}{p} \biggr)^{-1}
< \frac1{2^rx} \sum_{\substack{ 2^rx < f \leqslant 2^{r+1}x \\f\in\cF}} 1 \ll \frac1{2^rx} (2^{r+1}x)^{1/2} \ll 2^{-r/2}x^{-1/2},
\]
and consequently
\[
\sum_{\substack{f>x\\f\in\cF \\\omega_A(f)=m}} \frac{1}{f} \prod_{p\mid f} \biggl(1+\frac{1}{p} \biggr)^{-1} = \sum_{r=0}^\infty \sum_{\substack{ 2^rx < f \leqslant 2^{r+1}x \\f\in\cF \\\omega_A(f)=m}} \frac{1}{f} \prod_{p\mid f} \biggl(1+\frac{1}{p} \biggr)^{-1} \ll \sum_{r=0}^\infty 2^{-r/2}x^{-1/2} \ll x^{-1/2}. \qedhere
\]
\end{proof}

\begin{proof}[Proof of Theorem~\ref{main_result_3}]
Fix a sequence $A=(0,a_2,a_3,\ldots)$ of complex numbers and a target $m\in\bC$.
Every positive integer $n$ can be written uniquely as $n=qf$ where~$q$ is squarefree,~$f$ is powerful, and $(q,f)=1$ (indeed,~$q$ is the product of the primes dividing~$n$ exactly once). In this notation, the condition $\omega_{A}(n)=m$ is equivalent to $\omega_{A}(f)=m$ (since $a_1=0$), and thus
\begin{align} \label{decomposition}
\#\cN_{A,m}(x)=\sum_{\substack{f\leqslant x\\f\in\cF \\\omega_A(f)=m}} \sum_{\substack{q\leqslant x/f \\(q,f)=1}} \mu^2(q).
\end{align}
To estimate the inner sum above, we use~\cite[Lemma 2.17]{MV} which says that for any $y\geqslant 1$ and any positive integer~$f$,
\[
\sum_{\substack{n\leqslant y\\(n,f)=1}} \mu^2(n)=\frac{6}{\pi^2}y \prod_{p\mid f} \biggl(1+\frac{1}{p} \biggr)^{-1} + O\biggl(y^{1/2} \prod_{p\mid f} (1-p^{-1/2})^{-1} \biggr).
\]
Inserting this asymptotic formula into equation~\eqref{decomposition} yields
\begin{align*}
\#\cN_{A,m}(x) &= \frac{6}{\pi^2}x\sum_{\substack{f\leqslant x\\f\in\cF \\\omega_A(f)=m}} \frac{1}{f} \prod_{p\mid f} \biggl(1+\frac{1}{p} \biggr)^{-1} + O\biggl(x^{1/2} \sum_{\substack{f\leqslant x\\f\in\cF \\\omega_A(f)=m}} \frac{1}{f^{1/2}} \prod_{p\mid f}(1-p^{-1/2})^{-1} \biggr) \\
&= \frac{6}{\pi^2}x \biggl( \sum_{\substack{f\in\cF \\\omega_A(f)=m}} \frac{1}{f} \prod_{p\mid f} \biggl(1+\frac{1}{p} \biggr)^{-1} - \sum_{\substack{f>x\\f\in\cF \\\omega_A(f)=m}} \frac{1}{f} \prod_{p\mid f} \biggl(1+\frac{1}{p} \biggr)^{-1} \biggr) + O(x^{1/2}\log x) \\
&= \frac{6}{\pi^2}x \biggl( \sum_{\substack{f\in\cF \\\omega_A(f)=m}} \frac{1}{f} \prod_{p\mid f} \biggl(1+\frac{1}{p} \biggr)^{-1} + O(x^{-1/2}) \biggr) + O(x^{1/2}\log x)
\end{align*}
by Lemmas~\ref{f<x error} and~\ref{f>x error}, which completes the proof of the theorem.
\end{proof}

With Theorem~\ref{main_result_3} now established, it is a simple matter to prove Theorem~\ref{main_result_4}, which will also establish the special case that is Theorem~\ref{main_result_2}.

\begin{proof}[Proof of Theorem~\ref{main_result_4}]
Fix a sequence $A=(0,a_2,a_3,\ldots)$ of nonnegative integers. Note that $\sum_{m=0}^{\infty}e_{A,m} = 1$ (by the argument in Remark~\ref{abs conv remark}), and therefore $\sum_{m=0}^{\infty}e_{A,m}z^{m}$ converges absolutely for any complex number~$z$ with $|z|\leqslant1$. By Theorem~\ref{main_result_3},
\begin{align*}
\sum_{m=0}^{\infty}e_{A,m}z^{m}&=\frac{6}{\pi^2} \sum_{m=0}^{\infty} z^m \sum_{\substack{f\in\cF \\\omega_A(f)=m}} \frac1{f} \prod_{p\mid f} \biggl(1+\frac{1}{p} \biggr)^{-1} = \frac{6}{\pi^2} \sum_{f\in\cF} \frac{z^{\omega_A(f)}}{f} \prod_{p\mid f} \biggl(1+\frac{1}{p} \biggr)^{-1}.
\end{align*}
Since
${z^{\omega_A(f)}}/{f} = \prod_{p^j\| f} {z^{\omega_A(p^j)}}/{p^j} = \prod_{p^j\| f} {z^{a_j}}/{p^j}$,
the right-hand side equals its Euler product
\begin{align*}
\frac{6}{\pi^2} \sum_{f\in\cF} \frac{z^{\omega_A(f)}}{f} \prod_{p\mid f} \biggl(1+\frac{1}{p} \biggr)^{-1} &= \frac{6}{\pi^2} \prod_p \biggl(1+\biggl(1+\frac{1}{p} \biggr)^{-1} \sum_{j= 2}^{\infty} \frac{z^{a_j}}{p^{j}} \biggr) \\
&= \prod_p \biggl(1-\frac{1}{p^2} \biggr) \biggl(1+\biggl(1+\frac{1}{p} \biggr)^{-1} \sum_{j= 2}^{\infty} \frac{z^{a_j}}{p^{j}} \biggr) \\
&=\prod_{p} \biggl(1-\frac{1}{p^2} + \biggl(1-\frac{1}{p} \biggr) \sum_{j= 2}^{\infty} \frac{z^{a_j}}{p^{j}} \biggr),
\end{align*}
which is equal to the right-hand side of equation~\eqref{generating_series_general}, thus establishing the theorem.
\end{proof}

\section{Decay rates of the densities} \label{section_for_main_results_2_and_4}

In this section, we deduce Corollary~\ref{main_cor_2} from Theorem~\ref{main_result_2} and Corollary~\ref{main_cor_4} from Theorem~\ref{main_result_4}. The key step is to give a proposition establishing the rate of growth of infinite products such as those appearing in Theorems~\ref{main_result_2} and~\ref{main_result_4}, which we do after the following simple lemma  for the prime-counting function $\pi(y)$ and its logarithmically weighted version $\theta(y)$.

\begin{lemma} \label{pi theta lemma}
$\pi(y)\log y - \theta(y) \sim y/\log y$ as $y\to\infty$.
\end{lemma}

\begin{proof}
\newcommand{\li}{\mathop{\rm li}}
By the prime number theorem,
\begin{align*}
\pi(y)\log y - \theta(y) &= \bigl( \li(y) + O(ye^{-c\sqrt{\log y}}) \bigr) \log y - \bigl( y + O(ye^{-c\sqrt{\log y}}) \bigr) \\
&= \biggl( \Big( \frac y{\log y} + \frac y{\log^2 y} + O\Big( \frac y{\log^3 y} \Big) \Big) + O\Big( \frac y{\log^3 y} \Big) \biggr) \log y - \biggl( y + O\Big( \frac y{\log^2 y} \Big) \biggr) \\
&= \frac y{\log y} + O\Big( \frac y{\log^2 y} \Big). \qedhere
\end{align*}
\end{proof}

\begin{proposition} \label{general decay prop}
Fix a real number~$\kappa>1$, and let $R(p)$ be a positive function defined on primes~$p$ such that $R(p) \sim p^{-\kappa}$ as $p\to\infty$. Define the function $P(x) = \prod_p \bigl(1 + R(p)x \bigr)$. Then $\log P(x) \asymp x^{1/\kappa}/\log x$ as $x\to\infty$.
\end{proposition}

\begin{proof}
All implicit constants in this proof may depend on~$R(p)$ and~$\kappa$.
Choose $p_0$ so that $\frac12p^{-\kappa} < R(p) < 2p^{-\kappa}$ for all $p>p_0$. We write
\begin{align}
\log P(x) &= \sum_{p\leqslant p_0} \log \bigl(1 + R(p)x \bigr) + \sum_{p_0<p\leqslant x^{1/\kappa}} \log \bigl(1 + R(p)x \bigr) + \sum_{p>x^{1/\kappa}} \log \bigl(1 + R(p)x \bigr) \notag \\
&= O(\log x) + \sum_{p_0<p\leqslant x^{1/\kappa}} \log \bigl(1 + R(p)x \bigr) + \sum_{p>x^{1/\kappa}} \log \bigl(1 + R(p)x \bigr),
\label{three sums}
\end{align}
since the number of terms in the first sum, and the largest value of $R(p)$ appearing in that sum, are both bounded in terms of the function~$R$.

In the first sum in equation~\eqref{three sums},
\[
\tfrac12p^{-\kappa}x < R(p)x < 1 + R(p)x < (x^{1/\kappa}/p)^\kappa + 2p^{-\kappa}x = 3p^{-\kappa}x.
\]
Therefore
\[
\sum_{p_0<p\leqslant x^{1/\kappa}} \log (\tfrac12 p^{-\kappa}x) \leqslant \sum_{p_0<p\leqslant x^{1/\kappa}} \log \bigl(1 + R(p)x \bigr) \leqslant \sum_{p_0<p\leqslant x^{1/\kappa}} \log (3p^{-\kappa}x).
\]
The right-hand inequality is the same as
\begin{align*}
\sum_{p_0<p\leqslant x^{1/\kappa}} \log \bigl(1 + R(p)x \bigr) &\leqslant ( \log x + \log 3 ) \bigl( \pi(x^{1/k}) - \pi(p_0) \bigr) - \kappa \bigl( \theta(x^{1/k}) - \theta(p_0) \bigr) \\
&= \kappa \bigl( \pi(x^{1/\kappa})\log(x^{1/\kappa}) - \theta(x^{1/\kappa}) \bigr) + \pi(x^{1/\kappa}) \log 3 + O(\log x) \\
&\sim (\kappa+\log3) \frac{x^{1/\kappa}}{\log(x^{1/\kappa})} = \kappa(\kappa+\log3) \frac{{x^{1/\kappa}}}{\log x}
\end{align*}
by Lemma~\ref{pi theta lemma}. By the same calculation with $\log\frac12$ in place of $\log3$,
\begin{align*}
\sum_{p_0<p\leqslant x^{1/\kappa}} \log \bigl(1 + R(p)x \bigr) &\geqslant\kappa(\kappa-\log2) \frac{{x^{1/\kappa}}}{\log x}
\end{align*}
(note that $\kappa-\log2>1-\log2$ is bounded away from~$0$).
We conclude that
\begin{equation} \label{mid conclusion}
\sum_{p_0<p\leqslant x^{1/\kappa}} \log \bigl(1 + R(p)x \bigr) \asymp \frac{{x^{1/\kappa}}}{\log x}.
\end{equation}

In the second sum in equation~\eqref{three sums},
\[
0 \leqslant \log(1+R(p)x) \leqslant R(p)x < 2p^{-\kappa} x,
\]
and thus by partial summation,
\begin{align*}
0 \leqslant \sum_{p>x^{1/\kappa}} \log \bigl(1 + R(p)x \bigr) &\leqslant \sum_{p>x^{1/\kappa}} 2p^{-\kappa} x \\
&= 2x \int_{x^{1/\kappa}}^\infty t^{-\kappa} \, d\pi(t) \\
&= 2x \biggl( \pi(t)t^{-\kappa}\bigg|_{x^{1/\kappa}}^\infty + \int_{x^{1/\kappa}}^\infty \kappa t^{-\kappa-1}\pi(t)\,dt \biggr).
\end{align*}
The boundary term is well defined (since $\kappa>1$) and negative, and thus by the prime number theorem,
\begin{align*}
0 \leqslant \sum_{p>x^{1/\kappa}} \log \bigl(1 + R(p)x \bigr) &\ll x \biggl( 0 + \int_{x^{1/\kappa}}^\infty t^{-\kappa-1} \frac{t}{\log t} \,dt \biggr) \\
&\ll \frac{x}{\log x} \int_{x^{1/\kappa}}^\infty t^{-\kappa} \,dt = \frac{x}{\log x} \frac{(x^{1/\kappa})^{1-\kappa}}{\kappa-1} \ll \frac{x^{1/\kappa}}{\log x}.
\end{align*}
The proposition now follows by combining these inequalities with equations~\eqref{three sums} and~\eqref{mid conclusion}.
\end{proof}

All that is left is to connect the rates of growth of the generating functions in Theorems~\ref{main_result_2} and~\ref{main_result_4} to the decay rate of their Maclaurin coefficients.
We use the following classical information about entire functions \cite[Definition~2.1.1 and Theorem~2.2.2]{Boas}:

\begin{definition} \label{order definition}
An entire function $f(z)$ is said to be {\em of order~$\rho$} if
\begin{align*}
\limsup_{r\to \infty} \frac{\log\log M_f(r)}{\log r}=\rho
\end{align*}
where $M_f(r)=\max_{|z|=r} |f(z)|$. It is {\em of finite order} if it is of order~$\rho$ for some $\rho\in\bR$.
\end{definition}

\begin{lemma} \label{lemma_for_the_order}
Let $f(z) = \sum_{m=0}^{\infty} b_mz^m$ be an entire function.
The function $f(z)$ is of finite order if and only if
\begin{align*}
\mu=\limsup_{\substack{m\to \infty \\ b_m\ne0}} \frac{m\log m}{\log(1/|b_m|)}
\end{align*}
is finite, and in this case $f(z)$ is of order~$\mu$.
\end{lemma}

\begin{proof}[Proof of Corollary~\ref{main_cor_2}]
Set
\begin{align*}
P(x)=\prod_{p} \biggl(1+\frac{p-1}{p^{k+1}} x \biggr) \quad\text{and}\quad Q(z) = \sum_{m=0}^{\infty}e_{k,m}z^{m} = \prod_{p} \biggl(1+\frac{p-1}{p^{k+1}} (z-1) \biggr).
\end{align*}
When $|z|=r$, note that
\[
|Q(z)| \leqslant \prod_{p} \biggl(1+\frac{p-1}{p^{k+1}} (|z|+1) \biggr) = P\big( r+1 \bigr);
\]
thus by Proposition~\ref{general decay prop} with $\kappa=k$ and $R(p) = (p-1)/p^{k+1}$,
\[
\log |Q(z)| \ll \frac{(r+1)^{1/k}}{\log(r+1)} \ll \frac{r^{1/k}}{\log r}.
\]
On the other hand, when $z=r>3$ is real, then
\[
\log |Q(r)| = \log P(r-1) \gg \frac{(r-1)^{1/k}}{\log(r-1)} \gg \frac{r^{1/k}}{\log r}
\]
again by Proposition~\ref{general decay prop}. Together these last estimates show that
$\log M_Q(r) \asymp {r^{1/k}}/{\log r}$, which implies that
\[
\limsup_{r\to \infty} \frac{\log\log M_Q(r)}{\log r} = \limsup_{r\to \infty} \frac{\log(r^{1/k}) - \log\log r + O(1)}{\log r} = \frac1k.
\]
In particular, $Q(z)$ has order $\frac1k$ by Definition~\ref{order definition}; consequently, by Lemma~\ref{lemma_for_the_order},
\begin{align*}
\limsup_{\substack{m\to \infty \\ e_{k,m}\ne0}} \frac{m\log m}{\log(1/|e_{k,m}|)}=\frac{1}{k}.
\end{align*}
We know that $e_{k,m}>0$ by Theorem~\ref{main_result_1}, and so
\begin{align*}
\frac{m\log m}{\log(1/e_{k,m})} \leqslant \frac{1}{k}+o(1)
\end{align*}
with asymptotic equality for infinitely many~$m$; we conclude that
\begin{align*}
e_{k,m} \leqslant m^{-(k-o(1))m}
\end{align*}
which completes the proof of the corollary.
\end{proof}

\begin{proof}[Proof of Corollary~\ref{main_result_4seo}]
The proof is the same as the proof of Corollary~\ref{main_cor_2}, except that $Q(z)$ is changed to each of the three products
\begin{align*}
\sum_{m=0}^{\infty}e_{{S},m}z^{m}&=\prod_{p} \biggl(1+\frac{z-1}{p^{2}} \biggr)
\\\sum_{m=0}^{\infty}e_{{E},m}z^{m}&=\prod_{p} \biggl(1+\frac{z-1}{p(p+1)} \biggr)
\\\sum_{m=0}^{\infty}e_{{O},m}z^{m}&=\prod_{p} \biggl(1+\frac{z-1}{p^{2}(p+1)} \biggr)
\end{align*}
in turn, with corresponding modifications to $P(x)$ and $R(p)$;  instead of with $\kappa=k$, the appeal to Proposition~\ref{general decay prop} is made with $\kappa=2$ in the first two cases and $\kappa=3$ in the last case, and the rest of the proof goes through in exactly the same way.
\end{proof}

\section{Numerical calculations of densities, expectations, and variances} \label{section_calculations}

We now describe how we used the generating functions in Theorem~\ref{main_result_2} to facilitate the calculation of the densities $e_{k,m}$ in Table~\ref{ekm table} to the indicated high level of precision. Our approach is based on observations of Marcus Lai (private communication).

\begin{proposition} \label{recursive general}
Let $P(z)$ be any function with Maclaurin series
\[
P(z) = \sum_{n=0}^\infty C(n) z^n,
\]
so that $C(n) = \frac1{n!} P^{(n)}(0)$ for every $n\ge0$. Define
\[
S(0,z) = \log P(z) \quad\text{and}\quad S(n,z) = \frac{d^n}{dz^n} S(0,z)
\]
for all $n\ge1$; and define $S(n) = S(n,0)$ and $\tilde S(n) = \frac1{(n-1)!} S(n)$. Then for any $n\ge1$,
\begin{align*}
P^{(n)}(z) &= \sum_{k=0}^{n-1} \binom{n-1}k P^{(k)}(z) S(n-k,z).
\end{align*}
In particular, for $n\ge1$,
\begin{equation} \label{1}
C(n) = \frac1n \sum_{k=0}^{n-1} C(k) \tilde S(n-k),
\end{equation}
so that for example
\begin{align*}
C(1) &=C(0)\tilde{S}(1)= P(0) \tilde S(1), \\
C(2) &=\frac{1}{2} \bigl(C(0)\tilde{S}(2)+C(1)\tilde{S}(1)\bigr) =\frac{P(0)}{2} \bigl( \tilde S(1)^2 + \tilde S(2) \bigr), \\
C(3) &=\frac{1}{3} \bigl(C(0)\tilde{S}(3)+C(1)\tilde{S}(2)+C(2)\tilde{S}(1)\bigr) = \frac{P(0)}{6} \bigl( {\tilde S(1)^3} + 3 \tilde S(1) \tilde S(2) + 2 \tilde S(3) \bigr).
\end{align*}
\end{proposition}

\begin{proof}
We first verify that
\[
P'(z) = P(z)\frac{P'(z)}{P(z)} = P(z) \frac d{dz} \log P(z) = P(z) \frac d{dz} S(0,z) = P(z) S(1,z),
\]
which is the case $n=1$ of the first identity. The general case of the first identity now follows from using the product rule $n-1$ times in a row on this initial identity $P'(z) = P(z) S(1,z)$. The second identity follows by plugging in $z=0$ into the first identity and recalling that $C(n) = \frac1{n!} P^{(n)}(0)$.
\end{proof}

We apply this recursive formula (with subscripts inserted throughout the notation for clarity) with $C(m) = e_{k,m}$, so that
\begin{equation} \label{genfn}
P_k(z) = \sum_{m=0}^\infty e_{k,m} z^m = \prod_p \biggl( 1 + \frac{(p-1)(z-1)}{p^{k+1}} \biggr)
\end{equation}
by Theorem~\ref{main_result_2}. We compute
\begin{align}
S_k(0,z) &= \sum_p \log \biggl( 1 + \frac{(p-1)(z-1)}{p^{k+1}} \biggr) \notag \\
S_k(1,z) &= \frac d{dz} \sum_p \log \biggl( 1 + \frac{(p-1)(z-1)}{p^{k+1}} \biggr) = \sum_p \biggl( z+\frac{p^{k+1}}{p-1}-1 \biggr)^{-1} \label{Sks} \\
S_k(n,z) &= \frac{d^{n-1}}{dz^{n-1}} \sum_p \biggl( z+\frac{p^{k+1}}{p-1}-1 \biggr)^{-1} = \sum_p (-1)^{n-1} (n-1)! \biggl( z+\frac{p^{k+1}}{p-1}-1 \biggr)^{-n}, \notag
\end{align}
so that
\begin{equation} \label{Skn formula}
\tilde S_k(n) =(-1)^{n-1} \sum_p \biggl( \frac{p^{k+1}}{p-1}-1 \biggr)^{-n}.
\end{equation}
Therefore equation~\eqref{1} becomes
\[
e_{k,m} = \frac1m \sum_{j=0}^{m-1} e_{k,j} \tilde S_k(m-j),
\]
and in particular we have
\begin{align*}
e_{k,0} &=\prod_{p} \biggl(1-\frac{p-1}{p^{k+1}} \biggr),  \\
e_{k,1} &= e_{k,0} \tilde S_k(1), \\
e_{k,2} &= \frac{e_{k,0}}{2} \bigl( \tilde S_k(1)^2 + \tilde S_k(2) \bigr), \\
e_{k,3} &= \frac{e_{k,0}}{6} \bigl( {\tilde S_k(1)^3} + 3 \tilde S_k(1) \tilde S_k2) + 2 \tilde S_k(3) \bigr).
\end{align*}

\begin{remark}
We coded these formulas, and the one in equation~\eqref{Skn formula}, into SageMath and calculated approximations to them where we truncated the infinite product and sums to run over primes~$p\leqslant10^7$, resulting in the densities appearing in Table~\ref{ekm table} (the cases $k=2,3,4,5$ and $m=0,1,2,3$). While we do not include a formal analysis of the error arising from these truncations, we have listed the densities to nine decimal places to display our confidence in that level of precision.
\end{remark}

Finally, we extract the expectations and variances of various limiting distributions from their generating functions by relating those quantities to derivatives of their generating functions.

\begin{proof}[Proof of Corollary~\ref{corollary_X}]
Let $X_k$ be the discrete random variable whose distribution is the same as the limiting distribution of $\omega_k(n)$; then the generating function of this distribution is the function $P_k(z)$ in equation~\eqref{genfn}.
Proposition~\ref{recursive general} and equations~\eqref{genfn}--\eqref{Sks} tell us that
\begin{align*}
P_k'(1) &= P_k(1) S_k(1,1) =1\cdot \sum_p \frac{p-1}{p^{k+1}} = \sum_p \frac{p-1}{p^{k+1}} \\
P_k''(1) &= P_k'(1) S_k(1,1) + P_k(1) S_k(2,1) \\
&= \sum_p \frac{p-1}{p^{k+1}} \cdot \sum_p \frac{p-1}{p^{k+1}} + 1\biggl( - \sum_p \biggl( \frac{p-1}{p^{k+1}} \biggr)^2 \biggr) = \biggl( \sum_p \frac{p-1}{p^{k+1}} \biggr)^2 - \sum_p \biggl( \frac{p-1}{p^{k+1}} \biggr)^2
\end{align*}
But now by standard results from probability~\cite[Chapter~XI, Theorems~2--3]{Feller},
\begin{align*}
\bE[X_k] &= P_k'(1) = \sum_p \frac{p-1}{p^{k+1}} \\
\sigma^2[X_k] &= P_k''(1)+P_k'(1)-P_k'(1)^2 = \sum_p \frac{p-1}{p^{k+1}} - \sum_p \biggl( \frac{p-1}{p^{k+1}} \biggr)^2
\end{align*}
which is equivalent to the statement of the corollary.
\end{proof}

\begin{remark}
As before, we used SageMath to calculate truncations of these infinite sums, running over primes~$p\leqslant10^7$, to generate the approximate expectations and variances listed in Table~\ref{E and var table} for $k=2,3,4,5$.
\end{remark}

\begin{proof}[Proof of Corollary~\ref{corollary_Y}]
Let $X_A$ be the random variable whose distribution is the same as the limiting distribution of $\omega_A(n)$.
Using the same approach starting from the generating function~\eqref{main_result_4}, we see that
\begin{align*}
\bE & [X_A] = P_A'(1) = \frac d{dz} \prod_{p} \biggl(1-\frac{1}{p^2}+\sum_{j=2}^{\infty}z^{a_j} \biggl(\frac{1}{p^{j}}-\frac{1}{p^{j+1}} \biggr)\biggr) \bigg|_{z=1} \\
&= \biggl\{ \prod_{p} \biggl(1-\frac{1}{p^2}+\sum_{j=2}^{\infty}z^{a_j} \biggl(\frac{1}{p^{j}}-\frac{1}{p^{j+1}} \biggr)\biggr) \sum_p \frac d{dz} \log \biggl(1-\frac{1}{p^2}+\sum_{j=2}^{\infty}z^{a_j} \biggl(\frac{1}{p^{j}}-\frac{1}{p^{j+1}} \biggr)\biggr) \biggr\} \bigg|_{z=1} \\
&= 1 \sum_p \biggl(0+\sum_{j=2}^{\infty} a_jz^{a_j-1} \biggl(\frac{1}{p^{j}}-\frac{1}{p^{j+1}} \biggr) \biggr) \bigg/ \biggl(1-\frac{1}{p^2}+\sum_{j=2}^{\infty}z^{a_j} \biggl(\frac{1}{p^{j}}-\frac{1}{p^{j+1}} \biggr) \biggr) \bigg|_{z=1} \\
&= \sum_p \sum_{j=2}^{\infty} a_j \biggl(\frac{1}{p^{j}}-\frac{1}{p^{j+1}} \biggr) \bigg/ 1
\end{align*}
as claimed.
\end{proof}

\section*{Acknowledgments}

We thank Marcus Lai and Paul P\'eringuey for helpful discussions related to the calculations in Section~\ref{section_calculations}, and we are grateful to the Pacific Institute for the Mathematical Sciences (PIMS) whose welcoming environment contributed to this project.
The first author was supported by a University of Lethbridge postdoctoral fellowship.
The second author was supported in part by a Natural Sciences and Engineering Council of Canada Discovery Grant.

\end{document}